\documentclass[12pt]{amsart}
\usepackage{a4wide}
\usepackage{dsfont}
\usepackage{listings}
\usepackage{hyperref} 
\usepackage{booktabs,setspace}

\newtheorem{theorem}{Theorem}
\newtheorem{lemma}{Lemma}

\newtheorem{proposition}{Proposition}

\newcommand{\eps}{{\varepsilon}}
\newcommand{\R}{{\mathbb R}}
\newcommand{\N}{{\mathbb N}}
\newcommand{\1}[1]{\mathds{1}_{\{{#1}\}}}
 \newcommand{\PP}[1]{\mathbb{P}\left[#1\right]}
 \newcommand{\PPc}[2]{\mathbb{P}\left[#1\,\middle|\,#2\right]}
\newcommand{\EE}[1]{\mathbb{E}\left[#1\right]}
\newcommand{\EEc}[2]{\mathbb{E}\left[#1\,\middle|\,#2\right]}

\newcommand*\diff{\mathop{}\!\mathrm{d}}

\title{Sharp Tail Bounds Beyond Twice the Mean}

\author[]{Philipp Strack$^{\dagger}$}
\author[]{Jannik M. Westermann$^{\ddagger}$}
\thanks{$^{\dagger}$Yale University. Email: \texttt{philipp.strack@gmail.com}. Philipp Strack gratefully acknowledges financial support by the Cowles Foundation.}
\thanks{$^{\ddagger}$Yale University. Email: \texttt{jmwesterm@gmail.com}}

\date{\today}

\begin{document}

\begin{abstract}
Let $X_1,\ldots,X_n$ be independent nonnegative random variables with
$\EE{X_i}\le1$. We prove that, for every $t\ge2n+1$,
\[
    \PP{\sum_{i=1}^nX_i\ge t}
    \le 1-\left(1-\frac1t\right)^n.
\]
To prove this, we consider a relaxed optimization problem over a set of sequences of ordered, but non-independent random variables. This allows us to reformulate it recursively as a dynamic programming problem. 
The bound becomes an equality for the binary i.i.d.~random variables satisfying $\PP{X_i=0}= 1-\frac{1}{t}$ and $\PP{X_i=t}=\frac{1}{t}$, which remains a maximizer in the relaxed problem.
\end{abstract}

\maketitle

\section{Introduction}

Let $X_1,\ldots,X_n$ be independent nonnegative random variables satisfying
$\EE{X_i}\le 1$. We study the extremal problem
\[
    \sup_{X} \PP{X_1+\cdots+X_n\ge t},
\]
where the supremum ranges over all such collections. Markov's inequality gives
the universal bound $n/t$, but this bound does not use independence. A natural
independent construction is instead
\[
 X_i=\begin{cases}
 t,&\text{with probability }1/t,\\
 0,&\text{with probability }1-1/t.
 \end{cases}
\]
For this construction the event that the sum reaches $t$ is the event that at
least one variable takes its upper value, and hence
\[
    \PP{X_1+\cdots+X_n\ge t}
    =1-\left(1-\frac1t\right)^n.
\]
Our main result shows that this lower bound on the extremal probability is
sharp when $t\ge 2n+1$.

This question is a special case of a classical conjecture of Samuels on sums
of independent nonnegative random variables with prescribed means
\cite{Samuels_1966_independent_sum_conjecture}. To describe the conjectured
extremizers in the equal-mean case, suppose $t>n$ and fix $k\in\{1,\ldots,n\}$.
Let $n-k$ variables be identically equal to $1$, and let each of the remaining
$k$ variables independently take the value $t-n+k$ with probability
$1/(t-n+k)$ and the value $0$ otherwise. This gives the tail probability
\[
    H_k(t):=1-\left(1-\frac{1}{t-n+k}\right)^k.
\]
The equal-mean specialization of Samuels' conjecture asserts that the maximal
tail probability is $\max_{1\le k\le n}H_k(t)$. The construction above
corresponds to $k=n$.

Markov's inequality settles the case $n=1$ and $n=2$ was solved by Hoeffding and Shrikhande \cite{hoeffding1955bounds}.
Samuels proved his conjecture for $n=3,4$
\cite{Samuels_1966_independent_sum_conjecture,samuels1968more}, and, when $n\ge 5$, in the far-tail range
$t\ge n(n-1)$ \cite{samuels1969markov}. 

For \emph{identically} distributed summands, it can be shown that the $k=n$ construction is optimal for $t\ge\tfrac{5n-2}{3}$, see Frankl and Kupavskii \cite[Corollary 26]{frankl2022erdHos} who use connections to fractional matching in hypergraphs. These were established by Alon et.~al.~\cite{alon2012large} and used by {\L}uczak, Mieczkowska and {\v{S}}ileikis \cite{luczak2017maximal} to establish the result for $t\ge 2n-1$.

Our theorem treats arbitrary
independent, \emph{non-identically distributed} summands and lowers the previously known quadratic sufficient
threshold to a linear one, like there is in the i.i.d.~case.

The main idea of our proof is to enlarge the class of admissible random variables until the
optimization becomes recursive, while retaining enough of the structure of
independence to obtain a sharp bound. We first reduce to two-point
distributions, shift their lower support points to zero, and order their upper
support points decreasingly. Only then do we remove independence. The resulting
adapted processes have a predictable mean budget and conditionally binary
jumps whose possible sizes decrease over time. This last monotonicity condition
is essential: without it, the relaxed problem recovers the Markov bound. The
relaxed problem admits a dynamic programming formulation, which we control by
combining a binomial concentration bound for small jumps with an explicit
supersolution for large jumps.

Among further work published related to Samuels' conjecture there was interest in the special case of $t=n+1$, started by Feige \cite{feige2006sums} after whom this special case of Samuels' conjecture is named and subsequent improvements by He, Zhang and Zhang \cite{he2010bounding} and Garnett \cite{garnett2020small}.

\subsection{Concurrent Work}
Recently, independent of our analysis, Ling \cite{ling2026samuels} has resolved Samuels' conjecture. After the reduction to a sum of binary distributions $T$ and the threshold normalized to 1, the value $\EEc{T}{T<1}$ is exploited. A recursion is set up via
\[
\PP{T<1}=\PP{T_{-i}<1}(1-\PPc{T\ge 1}{T_{-i}<1})
\]
where $T_{-i}$ is the sum of the $n-1$ variables excluding $i$. The value $\PP{T_{-i}<1}$ is bounded recursively and the value $\PPc{T\ge 1}{T_{-i}<1}$ is controlled using the difference $\EEc{T}{T<1}-\EEc{T_{-i}}{T_{-i}<1}$. This difference is positive for the correct choice of $i$. The recursion allows bounding $\PP{T<1}$ by a deterministic function depending on $\EEc{T}{T<1}$ over which optimization occurs for the final bound. Han \cite{Han2026} provides an alternative presentation of Ling's proof.

Furthermore, independent of our analysis, Fu et al. \cite{fu2026sharp} have recently also shown that for $t\ge n+1$, it holds $\PP{\sum_{i=1}^n X_i\ge t} \le 1-(1-\tfrac{1}{t})^n$ for $X_1,\dots, X_n$ independent nonnegative and mean at most one random variables.
The proof in Fu et al. is based on a conjecture of Gaffke \cite{gaffke2005three}, recently proved by Vlassis and Thomas \cite{vlassis2026exact}.
The result is: Let $K(x) =\PP{\sum_{i=1}^n (x_i-1)E_i\le E_0}$ where $E_0,\dots, E_n\sim^{i.i.d.}\mathrm{Exp}_1$ are exponential 1 distributed. Then for $X_1,\dots,X_n$ independent, nonnegative with mean at most one and every $\alpha\in [0,1]$ it holds 
\[
\PP{K(X)\le \alpha} \le \alpha.
\]
Now $K$ can be interpreted as a partial volume of the standard $n$-simplex. Fu et al.~ combine Gaffke's conjecture with a generalization by Letwin and Yaskin \cite{letwin2024generalization} of Gr\"unbaum's inequality \cite{grunbaum1960partitions} for volumes of intersections of half-spaces with convex bodies. This gives the desired formula. 

Also based on these results, there appeared a proof of Feige's conjecture, i.e.~the special case $t=n+1$, by Nie and Wei \cite{nie2026feige}.

\subsection{Remark on Hypothesis testing}
Consider the composite null hypothesis that $X_1,\ldots,X_n$ are independent
and nonnegative and satisfy $\EE{X_i}\le1$ for every $i$. For
$0<\alpha\le1/3$, the test that rejects when $\bar{X}_n := \frac{1}{n}\sum_{i=1}^nX_i\ge \bar{t}_{\alpha,n}$ where the critical threshold is given by
\[
    \bar{t}_{\alpha,n}:=\frac{1}{n(1-(1-\alpha)^{1/n})},
\]
has size of at most $\alpha$ (i.e.~probability of falsely rejecting of at most $\alpha$). This follows as $\alpha \leq 1/3$ implies $n\bar{t}_{\alpha,n} \geq 2 n + 1$.\footnote{Bernoulli's inequality gives
$(1+1/(2n))^n\ge3/2$, so
$(1-\alpha)^{1/n}\ge(2/3)^{1/n}\ge2n/(2n+1)$ and hence
$n\bar{t}_{\alpha,n}\ge2n+1$.}
Without assuming that $X_1,\ldots,X_n$ are independent, the best possible threshold one can achieve is given by Markov's inequality and thus equal to\footnote{To see this denote $b_i=n/\alpha-(n-i)$ and consider $X_i$ which is distributed according to $\frac{1}{b_i} \delta_{b_i} + (1-\frac{1}{b_i}) \delta_0$ if $0=X_1= X_2=\ldots=X_{i-1}$ or $i=1$ and otherwise is equal to $1$ with probability $1$. These $X$ have the property that $\bar{X}_n \in \{0, 1/\alpha\}$ a.s.~which as $\mathbb{E}[\bar{X}_n]=1$ implies that $\mathbb{P}[\bar{X}_n=\frac{1}{\alpha}]=\alpha$.}
\[
    \bar{t}_{\alpha,n}^{Markov} := \frac{1}{\alpha} \,.
\]
Table~\ref{tab:comparison_Markov} provides a comparison of the threshold applied to the mean for small samples for tests at the $5\%$ and $10\%$ level.
As one can see, assuming that the draws are independent, or even assuming i.i.d.~leads only to a small improvement over not making any assumption on the correlation structure.
Taylor expanding $\bar{t}_{\alpha,n}$ yields that
\[
 \bar{t}_{\alpha,n}
 =\frac1{|\log(1-\alpha)|}+\frac1{2n}+O(n^{-2}),
 \qquad
 \frac{\bar{t}_{\alpha,n}}{\bar{t}^{Markov}_{\alpha,n}}
 \longrightarrow\frac{\alpha}{|\log(1-\alpha)|}.
\]
Moreover, as $n \mapsto \bar{t}_{\alpha,n}$ is decreasing, this bound holds for all $n$ and as $\alpha / |\log(1-\alpha) | = 1 - \alpha/2 - \alpha^2 / 12 + O(\alpha^3)$ this establishes that assuming i.i.d. can for any $n$ only lead to small improvements (of around $\alpha/2$) over the Markov bound.
\begin{table}[t]
\centering
\scriptsize
\setlength{\tabcolsep}{2.4pt}

\label{tab:critical_values}
\begin{tabular}{@{}l*{11}{r}@{}}
\toprule
&\multicolumn{11}{c}{$n$}\\
\cmidrule(l){2-12}
Method &1&2&3&4&5&6&7&8&9&10&$\infty$\\
\midrule
\multicolumn{12}{@{}l}{$\alpha=0.05$}\\
Sharp                   &20.000&19.747&19.663&19.621&19.596&19.579&19.567&19.558&19.551&19.546&19.496\\
Markov                  &20.000&20.000&20.000&20.000&20.000&20.000&20.000&20.000&20.000&20.000&20.000\\

\addlinespace
\multicolumn{12}{@{}l}{$\alpha=0.10$}\\
Sharp                   &10.000& 9.743& 9.659& 9.617& 9.592& 9.575& 9.563& 9.554& 9.547& 9.541& 9.491\\
Markov                  &10.000&10.000&10.000&10.000&10.000&10.000&10.000&10.000&10.000&10.000&10.000\\
\bottomrule
\end{tabular}
\caption{\label{tab:comparison_Markov}Critical values $\bar{t}_{\alpha,n}$ for tests that reject when
$\bar{X}_n\ge \bar{t}_{\alpha,n}$.}
\end{table}

The intuitive reason why assuming i.i.d. on its own does not lead to major improvement over the Markov bound lies in the fact that the i.i.d. assumption does not keep the random variables $X_1,\ldots,X_n$ from becoming more and more dispersed. 

In fact the bound is attained for $X$ such that 
\[
    \mathbb{E}[|X_1|^k] = \frac{1}{(1-(1-\alpha)^{1/n})^{k-1}} =     \left(\frac{n}{|\log(1-\alpha)|}\right)^{k-1} + O(n^{k-2})\,,
\]
and thus all moments with $k \geq 2$ of $X_1$ explode when $n$ grows large.

\section{Main Result and Proof Strategy}

\begin{theorem}\label{theorem_main}
    Let $n\in\N$, and let $X_1,\ldots,X_n$ be independent nonnegative random
    variables satisfying $\EE{X_i}\le1$. Then, for every $t\ge2n+1$,
    \[
\PP{\sum_{i=1}^n X_i\ge t} \le 1-\left(1-\frac{1}{t}\right)^n.
\]

Equality is attained by i.i.d.~random variables satisfying
$\PP{X_i=0}=1-1/t$ and $\PP{X_i=t}=1/t$.
\end{theorem}

\subsection*{A reduction to independent random variables with binary support}

Adding
the deterministic quantity $1-\EE{X_i}$ to each summand shows that it suffices
to consider variables with mean exactly one.\begin{NoHyper}\footnote{Indeed,
setting $\widetilde X_i:=X_i+1-\EE{X_i}$ preserves independence and
non-negativity, gives $\EE{\widetilde X_i}=1$, and yields
$\widetilde X_i\ge X_i$ pointwise. Hence
$\PP{\sum_{i=1}^nX_i\ge t}\le
\PP{\sum_{i=1}^n\widetilde X_i\ge t}$.}\end{NoHyper} Standard extreme-point arguments
then reduce the optimization, without changing its supremum, to two-point
distributions; see, for example,
\cite{hoeffding_1955_extrema_expected_value,winkler}. Write the support of
$X_i$ as $\{a_i,b_i\}$ and set
\[
    Z_i:=X_i-a_i,\qquad m_i:=\EE{Z_i}=1-a_i,
    \qquad w:=\sum_{i=1}^n m_i.
\]
The variables $Z_i$ have lower support point zero, and
\[
 \left\{\sum_{i=1}^nX_i\ge t\right\}
 =\left\{\sum_{i=1}^nZ_i\ge t-n+w\right\}.
\]
Thus shifting the support creates a trade-off: a smaller threshold is purchased
at the cost of a smaller total mean budget $w$. 
We have thus transformed the problem into an equivalent one in which it is without loss to consider only binary random variables $Z_i$ with support $\{0,u_i\}$.
Since the $Z_i$ are independent,
we may reorder them so that their upper support points $u_i$ are decreasing.

\subsection*{Relaxing independence}
We next relax the problem by considering non-independent random variables.
Let $I$ denote the class of sequences of independent random variables and for $u\in(0,\infty]$ and $w\ge0$, let $M(u,w)$ be the
class of nonnegative random processes $Z=(Z_1,Z_2,\ldots)$ for which there are nonnegative random variables $(U_i)$ which are predictable in the natural filtration of $Z$ such that
\begin{enumerate}
    \item the conditional law of $Z_i$ given $Z^{i-1}:=(Z_1,\ldots,Z_{i-1})$ is supported on
    $\{0,U_i\}$;
    \item $U_1\le u$ and $U_{i+1}\le U_i$ almost surely;
    \item the predictable mean satisfies $\EEc{Z_i}{Z^{i-1}} \leq 1$ and
    $\sum_{i=1}^{\infty}\EEc{Z_i}{Z^{i-1}}\le w$ almost surely.
\end{enumerate}
The variable $u$ bounds the next possible jump, while $w$ is the remaining budget for future means.
Let $D_{\le1}$ be the set of independent, nonnegative sequences of random variables with mean at most $1$. 

\begin{proposition}[Relaxation]\label{Prop:relaxation}
For $n\in\N_0$ and $t\in\R$,
\[
\begin{split}
\sup_{(X_i)\in D_{\le 1}}\PP{\sum_{i=1}^n X_i\ge t} &= \sup_{w\in [0,n]}\,\,\,\sup_{(Z_i)\in I\cap M(\infty,w)}\PP{\sum_{i=1}^n Z_i \ge t-n+w}\\
&\le \sup_{w\in [0,n]}\,\,\,\sup_{(Z_i)\in M(\infty,w)}\PP{\sum_{i=1}^n Z_i \ge t-n+w}. 
\end{split}
\]
\end{proposition}
\begin{proof}The first equality holds following the argument given above, while dropping independence
in the second line gives an upper bound.
\end{proof}
Proposition~\ref{Prop:relaxation} relaxes exactly one constraint: the conditional distribution of each \(Z_i\) may now depend on the preceding outcomes. It retains the two properties inherited from the independent problem: the pathwise nesting of the support hulls \([0,U_i]\) and the total predictable mean budget. The results below show that this relaxation is tight in the range \(t\ge2n+1\).

The decreasing-jump condition imposed on the processes in $M(u,w)$ is what makes the removal of
independence useful. To see this, suppose it were omitted and $t>n$. As long as
all previous outcomes are zero, an adapted process could let $Z_i=t-n+i$ with
probability $1/(t-n+i)$ and let $Z_i=0$ otherwise; after the first non-zero
outcome it could set all remaining variables equal to one. Every conditional
mean is one, and the sum reaches $t$ if and only if some non-zero outcome occurs.
Consequently its tail probability is
\[
  1-\prod_{i=1}^n\left(1-\frac{1}{t-n+i}\right)=\frac nt,
\]
the Markov bound. Along the all-zero history, however, its possible jump sizes
increase. The monotonicity inherited from reordering the independent variables
rules out precisely this behavior.

\subsection*{A dynamic programming formulation}
To compute the upper bound obtained in Proposition \ref{Prop:relaxation}, define the
value function
\[
    V_n(x,u,w) := \sup_{(Z_i)\in M(u,w)}\PP{\sum_{i=1}^n Z_i \ge x}
\]
for $n\in\N_0$. Conditioning on the first binary outcome gives the following
Bellman recursion.

\begin{lemma}[Dynamic programming principle]\label{Lemma:dynamic_programming}
    The value function $V$ is the unique solution to the recursive equation 
    \[
V_n(x,u,w)= \sup_{\substack{0<z\le u\\ 0\le m\le \min\{1,w,z\}}} \frac{m}{z}V_{n-1}(x-z,z,w-m) + \left(1-\frac{m}{z}\right) V_{n-1}(x,z,w-m)
\]
with initial value $V_0(x,u,w) = \1{x\le 0}$.
\end{lemma}

Lemma~\ref{Lemma:dynamic_programming} is the standard dynamic programming principle pioneered by Bellman 
(see for example Proposition~1(a), and Corollary~5.2(a) in \cite{shreve_bertsekas_1978_finite_horizon}).
We provide a proof in Section~\ref{sec:proof-prop2} to keep our paper self-contained.
The recursive formulation of Lemma~\ref{Lemma:dynamic_programming} is simpler than the original optimization problem.
Once the continuation value $V_{n-1}$ is known, the step from $V_{n-1}$ to
$V_n$ requires only an optimization over the first binary random variable: its
support is $\{0,z\}$ with $0<z\le u$, and its mean $m$ is at most
$\min\{1,w\}$ (and necessarily at most $z$). The two continuation values are
then already encoded by $V_{n-1}$.

The continuous-budget analogue of the candidate tail probability is
\[
G_w(x):=1-\left(1-\frac{1}{x}\right)^w
\]
for $x\ge 1$.
The following proposition which is established in Section~\ref{sec:proof-prop2} gives a uniform estimate on the value function $V$.

\begin{proposition}\label{Prop:target_bound}
Let $w\ge 0$ and $x\ge 2w+1$ and $x>1$. Then for $n\in \N_0$ we have 
    \[
    V_n(x,\infty,w) = V_n(x,x,w)\le G_w(x).
    \]
\end{proposition}

Jumps larger than the target threshold $x$ can be truncated at $x$ without decreasing the tail
probability, which explains the equality in Proposition~\ref{Prop:target_bound}.

\begin{proof}[Proof of Theorem~\ref{theorem_main}]
Given Proposition \ref{Prop:target_bound} we can establish Theorem~\ref{theorem_main} by the chain
\begin{align*}
    \sup_{(X_i)\in D_{\le1}}
        \PP{\sum_{i=1}^nX_i\ge t}
    &\le \sup_{0\le w\le n}V_n(t-n+w,\infty,w)\\
    &\le \sup_{0\le w\le n}G_w(t-n+w)\\
    &\le G_n(t).
\end{align*}
Here, the first inequality follows from Proposition~\ref{Prop:relaxation} and the definition of the value function.
The second inequality is implied by Proposition~\ref{Prop:target_bound}, as if $t\ge2n+1$ and $w\le n$, then
$t-n+w\ge2w+1,$
so Proposition \ref{Prop:target_bound} applies. The last inequality follows as $w\mapsto G_w(t-n+w)$ is increasing on $[0,n]$, so its maximum is attained
at $w=n$, where it equals $G_n(t)$.\footnote{To see this recall that $G_s(x):=1-(1-1/x)^s$. Set $a:=t-n\ge1$ and
$F(s):=G_s(a+s)$. Then
\[
\frac{\diff}{\diff s}\log(1-F(s))
=\log\left(1-\frac1{a+s}\right)+\frac{s}{(a+s)(a+s-1)}
\le \log\left(1-\frac1{a+s}\right)+\frac1{a+s}\le0,
\]
because $s\le a+s-1$ and $\log(1-r)\le-r$. Hence $F$ is increasing, and
$F(w)\le F(n)=G_n(t)$. The endpoints follow by continuity.}
\end{proof}

Thus, to establish Theorem~\ref{theorem_main} it remains to establish Proposition~\ref{Prop:target_bound}. The
Bellman recursion requires bounds on $V_n(x,u,w)$ for all $u\le x$, not only at
$u=x$. We divide the possible jump bounds at
\[
    u_0:=\frac{x+1}{2}.
\]

\subsection*{Small-jump regime}

For $n\ge1$ and $u\le u_0$, monotonicity in the jump bound gives
\[
V_n(x,u,w)\le V_n(x,u_0,w)\le Q_{n,w}(x),
\]
where the last inequality is established in Proposition \ref{Prop:concentration_ineq_binomial} with
\[
Q_{n,w}(x):=
\frac{x+1}{x-1}\left(\frac{2w}{x+1}-1+
\left(1-\frac{2w}{n(x+1)}\right)^n\right).
\]
The underlying concentration argument compares a normalized sum with the sum over a Bernoulli process in convex order and tests it against $(y-1)_+$. This
function counts the number of successes beyond the first and is therefore
positive only when at least two successes occur. The method goes back to
Bentkus \cite{Bentkus}; related binomial comparisons for martingales were
developed by Pinelis, see e.g.~\cite{pinelis2006binomial}. Proposition
\ref{Prop:concentration_ineq_binomial} adapts these ideas to a pathwise predictable mean
budget.

\subsection*{Large-jump regime}

For $u_0\le u<x$, the relevant terms arise from repeated attempts to make a jump
of size $u$. After the first success, only the residual threshold $x-u$ remains,
and its probability can be bounded by Markov's inequality using the remaining
mean budget. For integer budget, the time of the first success has geometric
weights. Interpolating these weights continuously in the mean budget motivates
the function
\[
P(x,u,w):= \int_0^w \lambda p^{w-\ell}\min\left\{1,\frac{\ell}{x-u}\right\}\diff \ell,
\]
where $p=1-1/u$ and $\lambda=-\log p$. Equivalently,
$\lambda p^{w-\ell}\diff\ell$ is the continuous interpolation of the
first-success probability at the point when $\ell$ units of budget remain, and
$\min\{1,\ell/(x-u)\}$ is the Markov bound for reaching the residual threshold.

Four properties make this heuristic bound into a supersolution of the Bellman
recursion. First, the small-jump estimate fits below it at the interface:
\[
    Q_{n,w}(x)\le P(x,u_0,w).
\]
Second, $P(x,u,w)$ is increasing in $u$ on $[u_0,x)$. Third, writing
$r=w-m$, the key induction inequality is
\[
 \frac{m}{z}\min\left\{1,\frac{r}{x-z}\right\}
 +\left(1-\frac{m}{z}\right)P(x,z,r)
 \le P(x,z,w).
\]
It controls the success branch of the Bellman recursion by Markov's inequality
and the failure branch by induction. Finally,
\[
    \sup_{u_0\le u<x}P(x,u,w)=G_w(x).
\]
The endpoint $u=x$ is handled by the corresponding one-step inequality for
$G$. These facts yield Proposition \ref{Prop:target_bound}, uniformly in the
number of remaining periods. The proof below follows this order: relaxation and
dynamic programming, the small-jump concentration inequality, comparison of
$Q$, $P$, and $G$, and finally the Bellman induction.

\section{Proofs of the Remaining Results}\label{sec:proof-prop2}

Proposition \ref{Prop:relaxation} has been proven in the text. 

\subsection{Dynamic programming}
\begin{proof}[Proof of Lemma \ref{Lemma:dynamic_programming}]
    The uniqueness follows inductively immediately. The initial condition holds since the empty sum is $0$. Assume the result holds for $n-1$. Now let $(Z_i)\in M(u,w)$. Denote by $z>0$ the upper value of the support of binary variable $Z_1$ and by $m=\EE{Z_1}$ its mean, such that 
\[
\PP{Z_1=z}=\frac{m}{z}, \quad \PP{Z_1=0}=1-\frac{m}{z}
\]
and $0<z\le u$ and $0\le m\le \min\{1,w,z\}$. Since the first mean is $m$, conditionally on either value of $Z_1$, the future process $(Z_i)_{i\ge 2}$ has predictable mean budget at most $w-m$. Since the upper values in the predictable supports are decreasing, all upper support points can be at most $z$. Thus for $y\in \{0,z\}$
\[
\PPc{\sum_{i=2}^n Z_i\ge x-y}{Z_1=y}\le V_{n-1}(x-y,z,w-m).
\]
Therefore conditioning on the outcomes of $Z_1$ we get 
\[
\begin{split}
    \PP{\sum_{i=1}^n Z_i \ge x}&= \left(1-\frac{m}{z}\right) \PPc{\sum_{i=2}^n Z_i \ge x}{Z_1=0}+ \frac{m}{z}\PPc{\sum_{i=2}^n Z_i\ge x-z}{Z_1=z}\\
    &\le\left(1-\frac{m}{z}\right) V_{n-1}(x,z,w-m)+ \frac{m}{z}V_{n-1}(x-z,z,w-m).\\
\end{split}
\]
Since $Z_i$ was chosen arbitrarily, this proves the upper bound. 

Conversely, fix $z$ and $m$ and choose $\eps$-suboptimal processes for the two values $V_{n-1}(x,z,w-m)$ and $V_{n-1}(x-z,z,w-m)$. Let $Z_1$ be $z$ with probability $\frac{m}{z}$ and 0 otherwise. Let the two processes be the continuations conditionally on the corresponding values of $Z_1$. The resulting process is in $M(u,w)$ and achieves the right hand side up to $\eps.$ Taking $\eps \to 0$ proves the reverse inequality.
\end{proof}

\subsection{Concentration inequality}

\begin{proposition}\label{Prop:concentration_ineq_binomial}
Let $(Y_i)$ be a stochastic process taking values in $[0,1]$.
Assume that its predictable mean budget is at most $a\in[0,1]$, that is,
$\sum_{i=1}^\infty\EEc{Y_i}{Y^{i-1}}\le a$. Then, for every $n\in\N$, $y>1$
\[
\PP{\sum_{i=1}^n Y_i\ge y}
\le \frac{1}{y-1}\left(a-1+
\left(1-\frac{a}{n}\right)^n\right).
\]
In particular, if $u\le u_0:=\tfrac{x+1}{2}$, $w\ge 0$ and $x\ge2w-1$, $x>1$ then
\[
V_n(x,u,w)\le
\frac{x+1}{x-1}\left(\frac{2w}{x+1}-1+
\left(1-\frac{2w}{n(x+1)}\right)^n\right)
=Q_{n,w}(x).
\]
\end{proposition}

\begin{proof}[Proof of Proposition \ref{Prop:concentration_ineq_binomial}]
Let $\xi_1,\dots,\xi_n$ be i.i.d.~uniformly distributed on the unit interval and independent of the sequence $(Y_i)$. Let
\[
B_i:= \1{\xi_i\le Y_i},\quad B=\sum_{i=1}^n B_i,\quad S= \sum_{i=1}^n Y_i.
\]
Notice that 
\[
\PPc{\xi_i\le Y_i}{Y_i}= Y_i.
\]
Thus
$\EEc{B}{Y^n}=S$ and hence $\EEc{B}{S}=S$, i.e.~ $B$ is a mean-preserving spread of $S$. In other words, $B$ dominates $S$ in the convex order. For the convex function $y\mapsto (y-1)_+$ we have by the conditional Jensen's inequality
\[
\EE{(S-1)_+}\le \EE{(B-1)_+}.
\]

Let
\[
p_i:=\PPc{B_i=1}{Y^{i-1},B^{i-1}}.
\]
Then independence of the $\xi_i$ gives
\[
p_i=\EEc{Y_i}{Y^{i-1}}
\]
and predictable success budget
\[
\sum_{i=1}^np_i\le a.
\]
In particular, $p_1=\EE{Y_1}$ is deterministic.

Now set
\[
R_n(a):=a-1+\left(1-\frac{a}{n}\right)^n.
\]
We prove by induction on $n$ that every such Bernoulli process satisfies
\[
\EE{(B-1)_+}\le R_n(a).
\]
For $n=1$, both sides vanish. Suppose the assertion holds for $n-1$
and write
\[
B'=\sum_{i=2}^nB_i.
\]
For almost every value of $(Y_1,B_1)$, under the corresponding
conditional law, the remaining process has predictable success budget
at most $a-p_1$. Hence
\[
\EEc{(B'-1)_+}{Y_1,B_1} \le R_{n-1}(a-p_1) \text{ a.s.}
\]
Moreover, the tower property gives
\[
\EEc{B'}{Y_1,B_1} = \EEc{\sum_{i=2}^np_i}{Y_1,B_1} \le a-p_1 \text{ a.s.}
\]
Writing $(B-1)_+=B_1B' + (1-B_1)(B'-1)_+$, we obtain, using the above inequalities
\[
\begin{split}
\EE{(B-1)_+}&\le p_1(a-p_1) + (1-p_1)R_{n-1}(a-p_1)\\
&=a-1+(1-p_1)\left(1-\frac{a-p_1}{n-1}\right)^{n-1}\\
&\le a-1+\left(1-\frac{a}{n}\right)^n=R_n(a)
\end{split}
\]
where the last inequality follows pointwise from AM-GM applied to the $n$ numbers $1-p_1, 1-\tfrac{a-p_1}{n-1},\dots, 1-\tfrac{a-p_1}{n-1}$. This concludes the induction.

Note that as $\{\sum_{i=1}^n Y_i\ge y\}=\{S\ge y\}$ we have 
\[
(y-1)\1{\sum_{i=1}^n Y_i\ge y}\le (S-1)_+.
\]
Taking the expectation we get the first part of the proposition
\[
(y-1)\,\PP{\sum_{i=1}^n Y_i\ge y}\le \EE{(S-1)_+}\le \EE{(B-1)_+} \le R_n(a).
\]
For any admissible process $(Z_i)$ for $V_n(x,u_0,w)$, the result applies to $Y_i= Z_i/u_0$, $a= w/u_0$ and $y=x/u_0$, since $x\ge 2w-1$ guarantees $w\le u_0$, that is $a\le 1$. 
Thus
\[
V_n(x,u_0,w)\le \frac{u_0}{x-u_0}R_n\left(\frac{w}{u_0}\right) =\frac{x+1}{x-1}\left(\frac{2w}{x+1}-1+\left(1-\frac{2w}{n(x+1)}\right)^n\right)=Q_{n,w}(x).
\]
Since $V_n(x,u,w)\le V_n(x,u_0,w)$ for $u\le u_0$ we are done.
\end{proof}

\subsection{Order of the bounds $G\ge P \ge Q$}
First two auxiliary lemmas.

\begin{lemma}\label{Lemma:lambda_ineqs}
    Let $q\in (0,1)$, $p=1-q$ and $\lambda=-\log p$. Then
    \[
    \lambda\ge q.
    \]
    Moreover if $q=\tfrac{1}{x-s}$ then 
    \[
\lambda'(s)=\frac{q^2}{1-q},\qquad \frac{\lambda'(s)}{\lambda(s)}\le \frac{q}{1-q}=\frac{1}{x-s-1}.
\]

\end{lemma}

\begin{proof}
    Taking the logarithm of the inequality $1-q\le e^{-q}$ gives $\lambda\ge q$. If $q=\tfrac{1}{x-s}$ then $q'=q^2$ and thus $\lambda'(s)=\tfrac{q'}{1-q}=\tfrac{q^2}{1-q}$. Dividing by $\lambda\ge q$ gives the remaining claim.
\end{proof}

\begin{lemma}\label{Lemma:g_properties}
    Let 
    \[
    g(y):=\frac{y-1+e^{-y}}{y},\quad y>0.
    \]
    Then $0<g(y)<1$ and
\[
    g'(y)\ge 0 ,\qquad \frac{g'(y)}{g(y)}\le \frac1y, \qquad \frac{g'(y)}{1-g(y)}\le \frac12.
    \]
\end{lemma}

\begin{proof}
Notice that $g(y)= \int_0^1 (1-e^{-ty})\diff t\in (0,1)$ and $g'(y)= \int_0^1 t e^{-ty}\diff t\ge 0$.
Now it holds $ae^{-a} \le 1-e^{-a}$ for $a\ge 0$, thus $a=ty$ gives
\[
yg'(y)= \int_0^1 tye^{-ty}\diff t \le \int_0^1(1-e^{-ty})\diff t = g(y).
\]
Finally 
\[
\frac{g'(y)}{1-g(y)}= \frac{\int_0^1 te^{-ty}\diff t}{\int_0^1e^{-ty}\diff t}
\]
is an average with decreasing density and therefore positive and less than the average of the uniform distribution on the unit interval which is $\tfrac{1}{2}$.
\end{proof}

\begin{lemma}\label{Lemma:P>Q}
    For $w\ge 0$ and $x>1$, $x\ge 2w+1$, $n\in \N$ we have 
    \[
    Q_{n,w}(x)\le P(x,u_0,w). 
    \]
\end{lemma}

\begin{proof}
For $w=0$ the result is immediate so assume $w>0$.
    Let us denote $s=x-u_0$, $a=\frac{w}{u_0}$ as well as $q=\frac{1}{u_0}\in (0,1)$, $p=1-q$ and $\lambda=-\log p$. Since $x\ge 2w+1$  we have $s\ge w$, thus in the definition of $P$, the Markov bound is active on all of $[0,w]$. That is
    \[
    \begin{split}
    P(x,u_0,w)&= \int_0^w \lambda p^{w-\ell}\min\{1,\frac{\ell}{s}\}\diff \ell\\
    &=\frac{p^w}{s}\int_0^w \lambda p^{-\ell}\ell \diff \ell
    \end{split}
    \]
    Integration by parts shows the following equality for $\alpha,\beta$
    \begin{equation}\label{eq:integral_r}
    \int_\alpha^\beta \lambda p^{-\ell}\ell \diff \ell= p^{-\beta}\left(\beta-\frac{1}{\lambda}\right)-p^{-\alpha}\left(\alpha-\frac{1}{\lambda}\right).
\end{equation}    
Applying it yields since $p=e^{-\lambda}$ and using $g$ from Lemma \ref{Lemma:g_properties}
\[
P(x,u_0,w)= \frac{\lambda w-1+p^w}{\lambda s}=\frac{w}{s}g(\lambda w). 
\]
Since generally $\left(1-\frac{a}{n}\right)^n\le e^{-a}$, we have
\[
R_n(a)= a-1+\left(1-\frac{a}{n}\right)^n \le a-1+e^{-a}=ag(a). 
\]
Thus 
\[
Q_{n,w}(x) = \frac{u_0}{s} R_n(a) \le \frac{u_0}{s}ag(a)= \frac{w}{s}g(a).
\]
By Lemma \ref{Lemma:lambda_ineqs} we have $\lambda
\ge q=\frac{1}{u_0}$ and thus $\lambda w\ge a$. Since by Lemma \ref{Lemma:g_properties} $g$ is increasing, we have 
\[
Q_{n,w}(x)\le\frac{w}{s}g(a) \le \frac{w}{s}g(\lambda w)= P(x,u_0,w). \qedhere
\]
\end{proof}

\begin{lemma}\label{Lemma:P_increasing}
    Let $w\ge 0$ and $x>1$ and $x\ge 2w+1$. Then $P(x,u,w)$ is increasing in $u$ on $[u_0,x)$.
\end{lemma}

\begin{proof}
Write 
\[
s=x-u, 
\]
thus $0<s\le \frac{x-1}{2}$. We will prove that $P(x,x-s,w)$ is decreasing as a function of $s$. Write
\[
q=\frac{1}{u}=\frac{1}{x-s}, \quad p=1-q, \quad \lambda=-\log p
\]
and note that $q\in (0,1)$ since $u\ge u_0= \tfrac{x+1}{2}$ and $x>1$.

Recall that
\[
P(x,u,w) = \int_0^w \lambda p^{w-\ell} \min\left\{1,\frac{\ell}{s}\right\}\diff \ell.
\]
We show that $P$ is decreasing in $s$ in the cases $s\le w$ and $s\ge w$.

\textbf{Case 1:} $s\le w$. The minimum equals $\tfrac{\ell}{s}$ on $[0,s]$ and equals 1 on $[s,w]$. Hence by direct integration for the second term and \eqref{eq:integral_r} for the first term we get
\[
\begin{split}
P(x,u,w)&=  \frac{p^w}{s}\int_0^s \lambda p^{-\ell}\ell\diff \ell+ p^w\int_s^w \lambda p^{-\ell}\diff \ell\\&= 1-p^{w-s}\frac{1-p^s}{s\lambda}.
\end{split}
\]
Since $p=e^{-\lambda}$, this means
\[
1-P(x,u,w)=e^{-\lambda(w-s)}(1-g(s\lambda)).
\]
To show that $1-P$ is increasing in $s$, we show that its logarithmic derivative is positive. Indeed we have
\begin{equation*}
\begin{alignedat}{2}
\frac{\diff}{\diff s}\log(1-P(x,u,w))
&= -\lambda' (w-s)+\lambda - \frac{g'(s\lambda) }{1-g(s\lambda)}(\lambda's+\lambda)
    && \\
& \ge  \lambda-\lambda' (w-s) -\frac12(\lambda+s\lambda')
    &&\qquad \text{Lemma \ref{Lemma:g_properties}} \\
&=\frac{\lambda}{2}-\lambda'(w-\frac{s}{2})
    &&\\
&\ge\frac{q}{2}- \frac{q}{x-s-1}(w-\frac{s}{2}) 
    &&\qquad\text{Lemma \ref{Lemma:lambda_ineqs}}\\
&\ge 0
    &&\qquad q>0,\quad x\ge 2w+1.
\end{alignedat}
\end{equation*}

\textbf{Case 2:} $s\ge w$. If $w=0$, then $P(x,u,0)=0$ so there is nothing to prove. Assume $w>0$. Now the minimum equals $\frac{\ell}{s}$ on the whole interval $[0,w]$. Using \eqref{eq:integral_r} we get
\[
P(x,u,w)=\frac{p^w}{s}\int_0^w \lambda p^{-\ell} \ell\diff \ell= \frac{w}{s}g(\lambda w). 
\]
Taking the log derivative we have
\begin{equation*}
\begin{alignedat}{2}
\frac{\diff}{\diff s}\log P(x,u,w)
&=-\frac{1}{s}+ \frac{g'(\lambda w)}{g(\lambda w)} \lambda' w
    && \\
& \le -\frac{1}{s} + \frac{\lambda'}{\lambda} 
    &&\qquad \text{ Lemma \ref{Lemma:g_properties}}\\
& \le -\frac{1}{s} + \frac{1}{x-s-1} 
    &&\qquad \text{ Lemma \ref{Lemma:lambda_ineqs}}\\
& \le 0 
    &&\qquad s\le \frac{x-1}{2}.
\end{alignedat}
\end{equation*}
Since $P(x,u,w)$ is continuous at $s=w$, this concludes the proof. 
\end{proof}

\begin{lemma}\label{lemma:supP=G}
    Let $w\ge 0$ and $x\ge 2w+1$ and $x>1$. Then 
    \[
    \sup_{u\in [u_0,x)}P(x,u,w) = G_w(x).
    \]
\end{lemma}

\begin{proof}
By Lemma \ref{Lemma:P_increasing} the supremum over $[u_0,x)$ is the limit as $u\to x$. 
Recall that
\[
P(x,u,w) = \int_0^w -\log\left(1-\frac{1}{u}\right) \left(1-\frac{1}{u}\right)^{w-\ell} \min\left\{1,\frac{\ell}{x-u}\right\}\diff \ell.
\]
Now as $u\to x$ we have $1-\frac{1}{u}\to 1-\frac{1}{x}$ and $\min\{1,\tfrac{\ell}{x-u}\}\to 1$ for all $\ell>0$. Thus dominated convergence gives 
\[
P(x,u,w)\xrightarrow{u\uparrow x} \int_0^w -\log\left(1-\frac{1}{x}\right) \left(1-\frac{1}{x}\right)^{w-\ell} \diff \ell = 1-\left(1-\frac{1}{x}\right)^w=G_w(x).
\]
\end{proof}

\subsection{Induction}

\begin{lemma}\label{lemma:induction_step}
    Let $w\ge 0$ and $0\le m\le \min\{1,w\}$ and denote the remaining budget $r=w-m$. Let $x>1$ as well as $z\in [u_0,x)$. Then 
    \begin{equation}\label{eq:induction_P}
    \frac{m}{z}\min\left\{1,\frac{r}{x-z}\right\} +\left(1-\frac{m}{z}\right) P(x,z,r)\le P(x,z,w).
    \end{equation}
\end{lemma}

\begin{proof}
Denote $q=\frac{1}{z}\in (0,1)$, $p=1-q$ and $\lambda=-\log p$. 
Using the Bernoulli inequality for an exponent $m\in [0,1]$ we have
\begin{equation}\label{eq:bernoulli_ineq}
p^m=(1-q)^m\le 1-mq.
\end{equation}
Note that $h(\ell)=\min\{1,\tfrac{\ell}{x-z}\}$ is increasing in $\ell$. 
Thus after direct integration
\begin{equation}\label{eq:PR_le_fR}
    P(x,z,r)\le h(r) \int_0^r\lambda p^{r-\ell}\diff \ell= h(r)(1-p^r)\le h(r).
\end{equation}

Further we bound $P(x,z,w)$ by splitting the integration interval and noting that the first term is $p^{w-r}P(x,z,r)=p^mP(x,z,r)$ and for the second term we bound via $h$ increasing, that is 
    \[
P(x,z,w) = \int_0^r \lambda p^{w-\ell}h(\ell)\diff \ell +\int_r^w \lambda p^{w-\ell}h(\ell)\diff \ell \ge  p^m P(x,z,r)+(1-p^m)h(r). 
\]
The term in the second bracket of the following equation is the left term in \eqref{eq:induction_P}. The inequality 
\[
[p^mP(x,z,r)+(1-p^m)h(r)]-[(1-mq)P(x,z,r)+mqh(r)]= (1-p^m-mq)(h(r)-P(x,z,r))\ge 0
\]
follows from \eqref{eq:bernoulli_ineq} and \eqref{eq:PR_le_fR}.
Combining the last two inequalities thus completes the proof.
\end{proof}

\begin{lemma}\label{lemma:induction}
    Let $n\in \N_0$ and $w\ge 0$, as well as $x>1$ and $x\ge 2w+1$ and $u\in [u_0,x)$. Then 
    \[
    V_n(x,u,w)\le P(x,u,w)
    \]
\end{lemma}

\begin{proof}
    For $n=0$ we have $V_0(x,u,w)= \1{x\le0}=0$. Now for sake of induction assume the lemma holds for $n-1$. We have
    \[
    \begin{split}    
    &V_n(x,u,w)= \sup_{\substack{0<z\le u\\0\le m\le \min\{1,w,z\}}} \frac{m}{z}V_{n-1}(x-z, z, w-m)+\left(1-\frac{m}{z}\right)V_{n-1}(x,z,w-m)\\
    &= \max\Bigg\{V_n\left(x,u_0,w\right),\sup_{\substack{u_0<z\le u\\0\le m\le \min\{1,w,z\}}} \frac{m}{z}V_{n-1}(x-z, z, w-m)+\left(1-\frac{m}{z}\right)V_{n-1}(x,z,w-m)\Bigg\}.
    \end{split}
    \]
    By Proposition \ref{Prop:concentration_ineq_binomial} the first term is bounded by $Q_{n,w}(x)\le P(x,u,w)$ by Lemma \ref{Lemma:P>Q} and Lemma \ref{Lemma:P_increasing}.
    For $z\in (u_0,u]$ we get
    \[
    \begin{split}
&\frac{m}{z}V_{n-1}(x-z,z,w-m)+\left(1-\frac{m}{z}\right)V_{n-1}(x,z,w-m)\\
&\le \frac{m}{z}\min\left\{1,\frac{w-m}{x-z}\right\}+\left(1-\frac{m}{z}\right)P(x,z,w-m)\le P(x,z,w)\le P(x,u,w),
\end{split}
\]
where the first inequality uses the Markov inequality and induction hypothesis, the second inequality uses \eqref{eq:induction_P} and the last inequality uses Lemma \ref{Lemma:P_increasing}. This concludes the proof. 
\end{proof}

\begin{proof}[Proof of Proposition \ref{Prop:target_bound}]
Denote 
\[
K = \sup_{0\le m\le \min\{1,w\}} \frac{m}{x} + \left(1-\frac{m}{x}\right)V_{n-1}(x,x,w-m).
\]
Since $V_n(x,u,w)$ is clearly increasing in $u$, by Lemma \ref{lemma:induction} we have
\[
V_n(x,x,w)\le \max\{K,\sup_{u\in [u_0,x)}P(x,u,w)\}.
\]
Denoting $p=1-\tfrac{1}{x}\in (0,1)$ we have, using Bernoulli's inequality for the exponent $m\in [0,1]$,
\[
\frac{m}{x}+\left(1-\frac{m}{x}\right)G_{w-m}(x)= 1-\left(1-\frac{m}{x}\right)p^{w-m} \le 1-p^mp^{w-m} =G_w(x).
\]
We proceed by induction, the case $n=0$ is immediate.
By induction hypothesis we have $V_{n-1}(x,x,w-m)\le G_{w-m}(x)$ and thus 
$K \le G_w(x)$. Lemma \ref{lemma:supP=G} gives the inequality.

It remains to show $V_n(x,\infty,w)\le V_n(x,x,w)$. To this end for $Z\in M(\infty,w)$, consider $\Tilde{Z}_i= \min\{Z_i,x\}$ and note that $(\Tilde{Z}_i)\in M(x,w)$, and since the random variables are nonnegative, $\{\sum_{i=1}^n \Tilde{Z}_i \ge x\}=\{\sum_{i=1}^n Z_i \ge x\}$. This concludes the argument.
\end{proof}

\bibliographystyle{amsalpha}
\bibliography{bibliography.bib}
\end{document}